\documentclass[]{article}
\usepackage{lineno,hyperref}
\modulolinenumbers[5]
\usepackage{a4wide}
\usepackage{epsf}
\usepackage{rotating}
\usepackage{subfigure}
\usepackage{amsfonts}
\usepackage{amsmath}
\usepackage{amssymb}
\usepackage{amsthm}
\usepackage{bm}
\usepackage{epstopdf}
\usepackage{graphicx}
\usepackage{color}

\usepackage{cite}
\theoremstyle{plain}
\newtheorem{theorem}{Theorem}
\newtheorem{corollary}[theorem]{Corollary}
\newtheorem{lemma}[theorem]{Lemma}

\theoremstyle{definition}
\newtheorem{definition}[theorem]{Definition}

\theoremstyle{remark}
\newtheorem{remark}[theorem]{Remark}

\begin{document}

\title{\textsc{Elliptical Tempered Stable Distribution and Fractional Calculus}}

\author{\textsc{Hassan A. Fallahgoul${}$}\thanks{Universit\'{e} libre de Bruxelles, Solvay Brussels School of Economics and Management, ECARES. 50, Av Roosevelt CP114, B1050 Brussels, Belgium. Tel: +32(0)26502792; Fax: +32(0)26504218; \texttt{hfallahg@ulb.ac.be}}
\hspace{0.05cm} and \textsc{Young S. Kim${}$}
\thanks{Stony Broke University, College of Business;
\texttt{Aaron.Kim@stonybrook.edu}
}}
\maketitle

\vspace{-0.8cm}
\begin{abstract}
A definition for elliptical tempered stable distribution, based on the characteristic function, have been explained which involve a unique spectral measure. This definition provides a framework for creating a connection between infinite divisible distribution, and particularly elliptical tempered stable distribution, with fractional calculus. Finally, some analytical approximations for the probability density function of tempered infinite divisible distribution, which elliptical tempered stable distributions are a subclass of them, are considered.
\\[2mm]
\\[2mm]
\end{abstract}

\newpage

\section{Introduction}
Stable and tempered stable distributions, which are a subclass of Infinite Divisible Distributions (hereafter IDD), have been used for solving many practical problems (see, for example, \cite{Kim01}, \cite{Rachev01}, and \cite{Fallahgoul01} among others.). The application of them in finance back to work of Mandelbrot in \cite{Mandelbrot01, Mandelbrot02}. As a matter of fact, before his work the usual assumption was that the distribution of price changes in a speculative series is approximately normal distribution. In the other words, in the financial time series context the distribution of the innovations (white noise) in a model like autoregressive moving average of returns series was normal distribution. But, Mandelbort showed that the returns from financial assets have the properties like leptokurtosis and skewness which can not support by normality assumption. On the other hand, the central limit theorem property, which is used based on normality assumption, is a powerful instrument in financial time series. Therefore, Mandelbort suggested a model that have the central limit theorem property as well as be able to capture leptokurtosis and skewness of return series. He used the stable distribution, in name of stable paretion, to reach suitable model. Therefore, the stable and tempered stable distributions, as a heavy tailed distribution, have been became the most popular alternative to the normal distribution which has been rejected by numerous empirical studies (see, for example, \cite{Rachev01}, and \cite{Rachev02}, among others). 

It is well known that stable distributions have infinite moments for some amounts of index parameter. In fact, they have infinite $ q-th $ moment for all $ q\geq\alpha $. This properties can be a disadvantage of the stable distributions in practical area. For example, in the context of time series a conclusion of literature is that the asset returns follow a model that the tails of it is heavier than normal and thinner than stable distribution (see, \cite{Rachev01}). In order to overcome this disadvantage, the L\'{e}vy measure of the stable distributions is multiplying by a tempering function. By this procedure a new class will be obtained, named tempered stable, such that they are not only IDD but also finite moments for all orders.\footnote{See, \cite{Rachev01}.}

The closed-form formula for the Probability Density Function (hereafter PDF) and Comulative Distribution Function (hereafter CDF) of the stable and tempered stable distribution, in the univarite and multivariate case, is not available. But, the characteristic function (hereafter CF) in the favorable form is accessible, which is used broadly instead of related PDF and CDF. But also, in the multivariate context the problem is still remain and working with related CF is not easy. For example, generating the multivariate stable and tempered stable distribution via the standard methods as fast fourier transform is so difficult even for lower dimensions. 

Based on the computational facilities, the univariate stable and tempered stable distributions successfully have been used for practical problems (see, for example \cite{Kim01}, \cite{Rachev01}, \cite{Rachev02},  and \cite{Rachev03} among others.). But, regarding to the multivariate case just a little works are exist, albeit for multivariate stable distributions (see, for example \cite{Nolan01}, \cite{Press01}, \cite{Fallahgoul01} and \cite{Fallahgoul02}). This fact backs to the definitions of each of them. The definition of the both multivariate stable and tempered stable distributions is based on a complicated spectral measures, but since the CF of the earlier can be defined based on some projections,\footnote{For more details see \cite{taqqu}.} working with this distribution in practical area is more convenient than multivariate tempered stable distributions.  

In this article we propose a suitable form of multivariate tempered stable distribution, named Elliptical Tempered Stable (hereafter ETS) distribution, and introduce some analytic approximations of the PDF of Tempered Infinite Divisible (hereafter TID), which ETS distribution is a subclass of it. More precisely, the contribution of the paper is threefold. First,  a well-define definition for ETS distribution based on a unique spectral measure is introduced, and also some theoretical properties of it are considered. Second, this paper provides a framework for getting a connection between fractional calculus and TID. In fact, some fractional Partial Differential Equations (hereafter PDE) are introduced that the fundamental solution of them give the entire family of the PDF of TID. Last, the analytic approximations for the PDF of TID and ETS distribution are given.

After discussing the notation to be used in Section 2, the paper proceeds as follows. The definition of ETS distribution and the link between multivariate tempered stable distributions and fractional PDE whose solution gives nearly all the TID distributions is explained in Section 3. The analytic approximations for PDF of TID distribution and ETS distribution have been presented in Section 4.


\section{Notation and Basic Definitions}
Here we present the notation and some basic definitions used in this paper. First a word on notation. A random variable and vector will be shown in capital and capital bold character, respectively. And also, an element of a vector is shown by non-bold words but still it is capital. 

Furthermore, the inner product and corresponding norm in $ \mathbb{R}^n $ are defined as $\langle\textbf{s},\textbf{t}\rangle=\displaystyle\sum_{i=1}^{n}s_it_i$ and $\|\textbf{s}\|=\langle\textbf{s}\rangle^{1/2}=(\displaystyle\sum_{i=1}^{n}s_i^2)^{1/2}$, respectively, where $\textbf{s}=(s_1,s_2,\cdots,s_n)$ and $\textbf{t}=(t_1,t_2,\cdots,t_n)$ are arbitrary elements in  $\mathbb{R}^n$. Let $\mathbb{S}^n$ be the unite sphere in $\mathbb{R}^n$: $\mathbb{S}^n=\{X\in\mathbb{R}^n: |X|=1\}$. For real $x$, the function $sign(x)$ is defined as
$$
sign(x)=\left\{
  \begin{array}{ll}
    \frac{x}{|x|}, & \hbox{$x\neq0$;} \\
    0, & \hbox{$x=0$.}
  \end{array}
\right.
$$
Let $\textbf{x}=(x_1, x_2, \cdots, x_n)$ be a random vector in $\mathbb{R}^n$ and let
\begin{equation}\nonumber
  \phi_{\textbf{x}}(\textbf{u})= E\exp(i\langle\textbf{u},\textbf{x}\rangle)= E\exp(i\displaystyle\sum_{i=1}^{n}u_ix_i),
\end{equation}
denote its Characteristic Function (hereafter CF). $\phi_{\textbf{x}}(\textbf{u})$ is also called the joint CF of the random variables $x_1, x_2, \cdots, x_n$. For any set as $ B $, $ \textbf{1}_{B} $ ia an indicator function, which $ \textbf{1}_{B}=1 $ for $ x\in B $ and $ \textbf{1}_{B}=0 $ for $x\neq B $. Furthermore, $ =^d $ means the both sides have the same distribution, and $ a\wedge b=\min\left(a,b \right)  $.

In general, there are two different definitions for elliptical distributions that are equivalent:\footnote{See, [13].} (1) definition based on the stochastic representation, and; (2) definition based on the CF. The former is explained by four components: a vector of locations, a nonnegative random variable as ${\mathcal R}_{\alpha}$, a k-dimensional random vector uniformly distributed on ${{\mathbb S}}^{n{\rm -}{\rm 1}}$ as $U^{(k)}\ $ that stochastically independent of ${\mathcal R}_{\alpha}$, and a matrix as ${{\mathbf \Lambda }\in {\mathbb R}}^{n\times k}$.\footnote{The rank of matrix $ \boldsymbol{\Lambda} $ is equal $ k $.} In contrast, the definition based on CF has three components: a vector of locations, a dispersion matrix that reproduces the elliptical, and a density generator that controls the tail thickness. The connection between these two definitions backs to this fact that every affinely transformed spherical random vector is elliptically distributed. 

Numerous distributions that are relevant for theoretical and practical works can be easily defined based on the stochastic representation definition of elliptical distributions: Gaussian, Laplace, Student-t, elliptical stable distribution (and hence Cauchy with $\alpha=1 $), and Kotz among others. For example, Gaussian $N({\boldsymbol\mu },{\rm \ }{\boldsymbol\Sigma })$ and Laplace ${\rm L(}{\boldsymbol\mu }{\rm ,\ }\lambda {\rm ,}{\mathbf \ }{\boldsymbol\Sigma }{\rm )}$ distributions are obtained if ${\mathcal R}_{\alpha}{\rm =}\sqrt{{\chi }^2_n}\ $ and ${\mathcal R}_{\alpha}{\rm =}\sqrt{{\chi }^2_n}\times \sqrt{\varepsilon (\lambda )}$, respectively where $\varepsilon (\lambda )$ is an exponential random variable with parameter $\lambda $ and stochastically independent of ${\chi }^2_n$.\footnote{It should be noted that the parameter $ \alpha $ in random variable $ {\mathcal R}_{\alpha} $ controls the tail thickness property.}

Before proceeding, recall that there is a L\'{e}vy measure related to the stable and tempered stable random variables, because the both of them are an IDD. Also, since tempered stable random variable has been obtained via tilting an IDD (stable distribution) and on the other hand tilting an IDD density leads to tilting of corresponding L\'{e}vy measure. So, the L\'{e}vy measure of the tempered stable distribution has been obtained via tilting the L\'{e}vy measure of a stable distribution. It should be noted that each IDD has been uniquely defined by the triple that named the \textit{L\'{e}vy triple}. Also, the \textit{L\'{e}vy-Khinchin} representation of the CF of the IDD which is based on the the L\'{e}vy triple, is very useful in study the theoretical properties of the stable random vectors and specially, more valuable, for different kinds of the tempered stable random vectors.\footnote{There are two major definitions for multivariate tempered stable distributions: (1) Rosinski's \cite{Rosinski01} multivariate tempered stable, and; (2) Bianchi et al. \cite{Kim02} multivariate tempered stable which known as Tempered Infinite  Divisible (TID). } 

In addition, we will see that the ETSD is a subclass of the symmetric Multivariate Normal Tempered Stable (MNTS) distribution.\footnote{MNTS distribution is obtained by a simple extension of normal tempered stable distribution. If the MNTS is symmetric, then the related random vector have the stochastic representation of elliptical tempered stable distribution. On the other hands, each symmetric MNTS distribution is a TID distribution, so we can reach the CF of symmetric MNTS distribution based on the spectral measure if we be able to reach the CF of symmetric TID distribution, see \cite{Kim01}.} And on the other hands, a symmetric MNTS is a subclass of the TID. So,  we will continue this section with the structure of the L\'{e}vy measure of the TID, which plays important role in our paper.  But first, recall the structure of the IDD. 

The CF (\textit{L\'{e}vy-Khinchin} representation ) of an IDD $ \mu $ on $ \mathbb{R}^n $ can be written as $ \Phi(\textbf{u})=\exp\left( \Psi(\textbf{u})\right)  $ where
\begin{equation}
\Psi(\textbf{u})=-\dfrac{1}{2}\left\langle \textbf{u},\textbf{A}\textbf{u}\right\rangle +i\left\langle \textbf{b}, \textbf{u}\right\rangle +\int_{\mathbb{R}^n}\left( e^{i\left\langle \textbf{u}, \textbf{x}\right\rangle }-1-i\dfrac{\left\langle \textbf{u}, \textbf{x}\right\rangle }{1+|\textbf{x}|^2} \right) M(dx),
\end{equation}
$ \textbf{A} $ is a symmetric nonnegative $ n\times n $ matrix, $ \textbf{b}\in\mathbb{R}^n $, and $ M $ satisfies
\begin{equation}\nonumber
\int_{\mathbb{R}^n}\left(\parallel x\parallel^2\wedge 1 \right)M(dx)<\infty, \qquad M({0})=0.
\end{equation}
The measure $ \mu $ uniquely defined by the L\'{e}vy triple $ (\textbf{A, \textbf{b, M}}) $, and we write $ \mu\sim IDD(\textbf{A, \textbf{b, M}} )$.

We will finish this section via the definition of the TID based on the Levy measure.

Based on the Rosinski's definition for multivariate tempered stable distribution ( Definition 2.1 of \cite{Rosinski01}), the L\'{e}vy measure of a multivariate tempered stable random vector without the Gaussian part in the polar coordinates is 
\begin{equation}\label{eq03}
M(dr,du)=r^{-\alpha-1}q(r,u)dr\sigma(du),
\end{equation}
where $ q:(0,\infty)\times  \mathbb{S}^{n-1}\longmapsto (0,\infty) $ is a Borel function such that $ q(\cdot,u) $ is completely monotone with $ q(\infty,u)=0 $  for each $ u\in\mathbb{S}^{n-1} $, and known as \textit{tempering function}. Likewise, $ \sigma $ is a finite measure on $ \mathbb{S}^{n-1} $ and $ 0<\alpha<2 $. On the other hands, the L\'{e}vy measure of a stable distribution on $ \mathbb{R}^n $ in polar coordinate is
\begin{equation}
M_0(dr,du)=r^{-\alpha-1}dr\sigma(du).
\end{equation}

A tempered stable distribution is characterized by the spectral measure as defined in Definition 2.4 of \cite{Rosinski01}, which based on Theorem 2.3 of \cite{Rosinski01} the L\'{e}vy measure of the tempered stable distributions can be written in the form
\begin{equation}\label{equ04}
M(A)=\int_{\mathbb{R}^n}\int_{0}^{\infty}I_A(tx)t^{-\alpha-1}e^{-t}dtR(dx),\quad A\in\textbf{B}(\mathbb{R}^n),
\end{equation}
where $ R $ is the spectral measure of the tempered stable distribution on $ \mathbb{R}^n $ such that
\begin{equation}\nonumber
\int_{\mathbb{R}^n}\left(\parallel x\parallel^2\wedge \parallel x\parallel^{\alpha} \right)R(dx)<\infty, \qquad R({0})=0.
\end{equation}
Now, let $ \left\lbrace Q(\cdot,u)\right\rbrace_{u\in S^{n-1}}  $ be a measurable family of Borel measures on $ (0,\infty) $, then the tempering function $ q $ in (\ref{eq03}) can be represented as\footnote{Based on the structure of this function different kinds of distribution in the multivariate can be obtained. In fact, all of them (i.e., $ q(r,u) $ and $ q(r^2,u) $ are a subclass of Grabchak which the $ q $ is $ q(r^p,u)$ and $ p>0 $ (See, \cite{Grabchak}).}
\begin{equation}
q(r,u)=\int_{0}^{\infty}e^{-rs}Q(ds|u). 
\end{equation}
Also, let measure $ Q $  be as
\begin{equation}
Q(A)=\int_{\mathbb{S}^{n-1}}\int_{0}^{\infty}I_A(ru)Q(dr|u)\sigma(du),\quad A\in\textbf{B}(\mathbb{R}^n),
\end{equation}
then $ Q({0})=0 $. Now, by defining measure $ R $ as
\begin{equation}\label{eq04}
R(A)=\int_{\mathbb{R}^{n}}I_A(\frac{x}{\parallel x\parallel^2})\parallel x\parallel^{\alpha}Q(dx),\quad A\in\textbf{B}(\mathbb{R}^n),          
\end{equation}
we can see that measure $ R $ in equation (\ref{eq04}) satisfies in the conditions of the spectral measure of a tempered stable distribution.\footnote{See Theorem 2.3 \cite{Rosinski01}.}

Bianchi et al. \cite{Kim02} based on the Rosinski's approach in \cite{Rosinski01} define another class of multivariate tempered stable random vector, known as TID. In fact, they are modify the radial component of $ M_0 $ and obtain a probability distribution with lighter tails than stable ones. 
\begin{definition}(TID measure, see \cite{Kim02})
Let $ \mu $ be an IDD measure on $ \mathbb{R}^n $ without Gaussian part (\textbf{A}=0). $ \mu $ is called TID if its L\'{e}vy measure $ M$ can be written in the polar coordinates as equation (\ref{eq03}), where $ \alpha $ is a real number, $ \alpha\in[0,2) $, $ \sigma $ is a finite measure on $ \mathbb{S}^{n-1} $, and $ q:(0,\infty)\times  \mathbb{S}^{n-1}\longmapsto (0,\infty) $ is a Borel function as
\begin{equation}
q(r,u)=\int_{0}^{\infty}e^{-r^2s^2}Q(ds|u),
\end{equation}
 with $ \left\lbrace Q(\cdot,u)\right\rbrace_{u\in S^{n-1}}  $ be a measurable family of Borel measures on $ (0,\infty) $.
\end{definition}
\begin{remark}
The difference  between the Rosinski and TID definition is just related to the tempering function ($ q(r,u) $).
\end{remark}
\begin{definition}(CF of TID, see \cite{Kim02})\label{TID:CF01}
Let $ \mu $ be a TID distribution with L\'{e}vy measure given by \ref{equ04}, $ \alpha\in[0,2) $, and $ \alpha\neq1 $. If the distribution has finite mean, i.e., $ \int_{\mathbb{R}^n}\parallel \textbf{x}\parallel \mu(dx)<\infty$, then
\begin{equation}\label{equ07}
\Phi_{\textbf{X}}(\textbf{u})=E\exp \left\lbrace i\left\langle \textbf{u},\textbf{X}\right\rangle \right\rbrace =\exp\left\lbrace \int_{\mathbb{R}^n}\psi_{\alpha}(\left\langle \textbf{u},\textbf{X}\right\rangle )R(dx) +i\left\langle \textbf{u},\textbf{m} \right\rangle \right\rbrace, 
\end{equation}
where
\begin{eqnarray}
\nonumber\psi_{\alpha}(s)=2^{-\alpha/2-1}\left[ \Gamma\left( -\dfrac{\alpha}{2}\right) \left( _1F_1\left( -\dfrac{\alpha}{2},\dfrac{1}{2};\left( \dfrac{i\sqrt{2}s}{2}\right)^2 \right)-1 \right)\right.\\
\left.+i\sqrt{2}s\Gamma\left( \dfrac{1-\alpha}{2}\right) \left( _1F_1\left( \dfrac{1}{2}-\dfrac{\alpha}{2},\dfrac{3}{2};\left( \dfrac{i\sqrt{2}s}{2}\right)^2 \right)-1 \right) \right], 
\end{eqnarray}
and $ \textbf{m}=\int_{\mathbb{R}^n} x\zeta(dx)$, and $ _1F_1 $ is the \textit{Kummer or confluent hypergeometric function of first kind}. Furthermore, if $ 0<\alpha<1 $, then the CF can be written in the alternative form
\begin{equation}
\Phi_{\textbf{X}}(\textbf{u})=E\exp \left\lbrace i\left\langle \textbf{u},\textbf{X}\right\rangle \right\rbrace =\exp\left\lbrace \int_{\mathbb{R}^n}\psi_{\alpha}^0(\left\langle \textbf{u},\textbf{X}\right\rangle )R(dx) +i\left\langle \textbf{u},\textbf{m}_0 \right\rangle \right\rbrace, 
\end{equation}
where
\begin{eqnarray}
\nonumber\psi_{\alpha}(s)=2^{-\alpha/2-1}\left[ \Gamma\left( -\dfrac{\alpha}{2}\right) \left( _1F_1\left( -\dfrac{\alpha}{2},\dfrac{1}{2};\left( \dfrac{i\sqrt{2}s}{2}\right)^2 \right)-1 \right)\right.\\
\left.+i\sqrt{2}s\Gamma\left( \dfrac{1-\alpha}{2}\right) \left( _1F_1\left( \dfrac{1}{2}-\dfrac{\alpha}{2},\dfrac{3}{2};\left( \dfrac{i\sqrt{2}s}{2}\right)^2 \right) \right) \right]. 
\end{eqnarray}
\end{definition}

If random vector $ \textbf{X} $ be an TID, it will be given by $ \textbf{X}\sim TID_{\alpha}(R,\textbf{m}) $ and in the alternative form by $\textbf{X}\sim TID_{\alpha}^0(R,\textbf{m}_0) $ where $ \alpha\in(0,2) $ and $ \alpha\in(0,1) $, respectively. It should be noted that the spectral measure of the tempered stable distribution, $ R(dx) $, is unique.

\section{Elliptical Tempered Stable Distribution and Fractional PDE}
In this section, we provide the definition of the ETSD and then the link between fractional PDE and TID are presented. First, we provide the CF of the symmetric TID random vector.
\begin{lemma}\label{pro01}
$ \textbf{X}$ is symmetric TID random vector in $ \mathbb{R}^n $ with $ 0<\alpha<2 $ if and only if there exist a unique spectral finite measure such that
\begin{equation}\label{equ05}
 E\exp \left\lbrace i\left\langle \textbf{u},\textbf{X}\right\rangle \right\rbrace =\exp\left\lbrace 2^{-\alpha/2-1}\Gamma(-\dfrac{\alpha}{2}) \int_{\mathbb{R}^n} \left( _1F_1\left( -\dfrac{\alpha}{2},\dfrac{1}{2};\dfrac{-1}{2} \left\langle \textbf{u},\textbf{X}\right\rangle^2 \right)-1 \right) R(dx)  \right\rbrace. 
\end{equation}

\end{lemma}

\begin{proof}

The proof is straightforward. It comes from this fact that a necessary and sufficient condition for an arbitrary random vector $ \textbf{X}$ to be symmetric is that its CF be real. And also, by defining a symmetric measure as $ \frac{1}{2}\left(R(dx)+R(-dx) \right)  $ and replacing with measure $ R(dx) $ the appropriate result is obtained.

\end{proof}
\begin{remark}
In Lemma \ref{pro01}, random vector $ \textbf{X} $ is symmetric about zero. One can easily extend it to any point as \textbf{m} of $ \mathbb{R}^n $, which it is more convenience for the elliptical distributions.
\end{remark}
The following definition is a direct corollary of Lemma \ref{pro01}. Since, a random vector ${\mathbf X}$ is elliptically distributed if and only if there exist a vector $\textbf{m}$, a symmetric positive define matrix ${\mathbf \Sigma }$ and a function $\varphi {\rm :}{{\mathbb R}}^{{\rm +}}\mapsto {\mathbb R}$ such that the CF of ${\mathbf X}$ is of the form\footnote{See, \cite{Frahm}.} 
\begin{equation}\label{equ01}
\Phi_\textbf{X}\left({\textbf{u}}\right){\rm =exp}\left(i\left\langle \textbf{m}, \textbf{u} \right\rangle\right){\rm \times }\varphi {\rm (}\textbf{u}^T{\mathbf \Sigma }\textbf{u}{\rm ).}
\end{equation}
So, if random vector $ \textbf{X} $ be a symmetric TID, then the CF of symmetric TID (equation (\ref{equ05})) can be shown as equation (\ref{equ01}), where the symmetric positive define matrix ${\mathbf \Sigma }$ can be extracted from the unique spectral measure $R{\rm (}dx{\rm )}$. 
 
It should be noted that the extracting the symmetric positive define matrix ${\mathbf \Sigma }$, due to the complicated form of spectral measure $R{\rm (}dx{\rm )}$, is difficult. But, since each symmetric MNTS distribution is a subclass of the TID distribution, which it has the CF as Definition \ref{TID:CF01},\footnote{See \cite{Kim02}.} and also, every symmetric MNTS distribution is an elliptical distribution,\footnote{See \cite{Rachev03}.} the existence of this matrix in the mentioned form can be proved. 

\begin{definition}\label{def05}
A random vector $ \textbf{X} $ is said an elliptical tempered stable random vector if there exist a unique measure such that the CF of it be as equation (\ref{equ05}).\footnote{It should be noted, due to the uniqueness of the spectral measure, Definition \ref{def05} is well defined.} 
\end{definition}

Before explaining the connection between tempered stable random vectors and fractional PDE, some theoretical properties of ETSD  are given. The symmetric MNTS distribution, which is a subclass of TID distribution, with the CF as Definition \ref{TID:CF01}, is come from a stochastic process. By setting the time variable equal 1, one can get this distribution (symmetric MNTS) from related stochastic process.\footnote{See, \cite{Kim01}.} Also, since the stable and tempered stable distributions don't have the closed form formula for PDF and CDF, the structure of their CF plays important role in the practical area. On the other hands, since the stable and tempered stable distribution are a subclass of IDD, and also based on the \textit{L\`{e}vy-Khinchine} formula the CF of the IDD can be obtained. Therefore, the L\`{e}vy-Khinchine formula provides a very useful tool for studying theoretical properties of these distributions. Also, because the uniqueness of the CF of a distribution, we check some properties of ETSD based the CF of symmetric MNTS which is exist in the suitable form. 

The CF of the symmetric MNTS process can be obtained from related multivariate Lévy-Khinchine formula.\footnote{See, \cite{Kim01}} The CF of the symmetric MNTS process is given by 
\begin{equation} \label{GrindEQ__43_} 
{\Psi }_X{\rm (}u{\rm )=exp}\left[t{\rm \times }\left\{\frac{{\rm 2}\lambda }{\alpha }\left({\rm 1-}{\left({\rm 1-}\frac{{\rm 1}}{\lambda }\left({\rm -}\frac{{\rm 1}}{{\rm 2}}u^T\Sigma u\right)\right)}^{\frac{\alpha }{{\rm 2}}}\right){\rm +}i{\mu }^Tu\right\}\right], 
\end{equation} 
so, by replacing $t{\rm =1}$ in equation \eqref{GrindEQ__43_} we will get the CF of the ETSD as follow 
\begin{equation} \label{equ06} 
{\Psi }_X{\rm (}u{\rm )=exp}\left[\frac{{\rm 2}\lambda }{\alpha }\left({\rm 1-}{\left({\rm 1-}\frac{{\rm 1}}{\lambda }\left({\rm -}\frac{{\rm 1}}{{\rm 2}}u^T\Sigma u\right)\right)}^{\frac{\alpha }{{\rm 2}}}\right){\rm +}i{\mu }^Tu\right]. 
\end{equation}  
\begin{remark}
As mentioned before, both definitions of elliptical distribution are equivalent. So, from equation (\ref{equ06}) we will be able to reach to the stochastid representation definition of the elliptical distribution, which In the next theorem, we will used this definition.
\end{remark}

Corollary \ref{The06} is the answer of this question that is there any connection between the PDF of an ETS random vector with related Gaussian underlying vector? If the answer is yes, the next question is that what is this connection? Corollary \ref{The06} explains more details in this regard.  But, first the definition of a subordinator is need. 

Univariate L\'{e}vy processes $ S_t, t\geq0 $ on $ \mathbb{R} $ with almost surely non-decreasing trajectories are called \textit{subordinators}. Now, let $ \alpha\in(0,2) $ and $ \theta>0 $. The purely non-Gaussian infinite divisible random variable $ T $ whose CF is given by 
\begin{equation*}
\phi_N(u)=\exp\bigg( -\dfrac{2\theta^{1-\frac{\alpha}{2}}}{\alpha}((\theta-iu)^{\frac{\alpha}{2}}-\theta^{\frac{\alpha}{2}})\bigg), 
\end{equation*}
is referred to as the tempered stable subordinator with parameter $ (\alpha,\theta) $.

\begin{corollary}\label{The06}
Let \textbf{X} be an ETSD as 
\begin{equation}
\textbf{X}=^d\boldsymbol{\mu}+\textbf{C}\sqrt{T}\textbf{N},
\end{equation}
where ${\boldsymbol\mu }{\rm =(}{\mu }_{{\rm 1}},{\mu }_{{\rm 2}}{\rm ,}\cdots {\rm ,}{\mu }_n{\rm )}\in {{\mathbb R}}^n,{\mathbf C}{\rm =(}C_{{\rm 1}},C_{{\rm 2}}{\rm ,}\cdots {\rm ,}C_n{\rm )}\in {\rm (0,}\infty {{\rm )}}^n$ and $\textbf{N}=(N_1,N_2,\cdots,N_n)$ be a n-dimensional standard normal distributed with covariance matrix ${\left\{{\rho }_{r,s}\right\}}^n_{r,s{\rm =1}}$. And also, $T$ is the tempered stable subordinator and independent of ${\mathbf N}$. Then there is a one-to-one correspondence between the PDF of \textbf{N} and that of \textbf{X}.
\end{corollary}
\begin{proof}
Based on equation (\ref{equ06}), we will have
\begin{equation}\label{equ16}
\Psi_{X}(u) =\exp\left[ \dfrac{2\lambda}{\alpha}\left(1-\left(1-\dfrac{1}{\lambda}\left(-\dfrac{1}{2}\sum_{i=1}^{n}\sum_{j=1}^{n}u_iu_j\rho_{i,j} \right)  \right)^{\frac{\alpha}{2}}  \right) +i\mu^T u \right],  
\end{equation}
where $ \rho_{i,j}=Cov(N_i,N_j) $.

Now, let  $ \textbf{N}' $ be a standard normal vector on $ \mathbb{R}^n $ with covariance matrix $ \rho' $. We must show that if $ \textbf{X}=\boldsymbol\mu +C\sqrt{T}\textbf{N}  $ has the same distribution as $ \textbf{X}'=\boldsymbol\mu +C\sqrt{T}\textbf{N}'$, then $ \textbf{N}' $ has the same distribution as $ \textbf{N} $. If two distributions be same, then the CF of them are same. So, from Equation (\ref{equ16}) we will have
\begin{equation}\nonumber
\nonumber\Psi_{\textbf{X}}(u) =\Psi_{\textbf{X}'}(u), 
\end{equation}
or
\begin{eqnarray}\nonumber
\nonumber\exp\left[ \dfrac{2\lambda}{\alpha}\left(1-\left(1-\dfrac{1}{\lambda}\left(-\dfrac{1}{2}\sum_{i=1}^{n}\sum_{j=1}^{n}u_iu_j\rho_{i,j} \right)  \right)^{\frac{\alpha}{2}}  \right) +i\mu^T u \right]\\
\nonumber=\\
\nonumber\exp\left[ \dfrac{2\lambda}{\alpha}\left(1-\left(1-\dfrac{1}{\lambda}\left(-\dfrac{1}{2}\sum_{i=1}^{n}\sum_{j=1}^{n}u_iu_j\rho_{i,j}' \right)  \right)^{\frac{\alpha}{2}}  \right) +i\mu^T u \right], 
\end{eqnarray}
therefore
\begin{equation}\nonumber
\sum_{i=1}^{n}\sum_{j=1}^{n}u_iu_j\rho_{i,j}=\sum_{i=1}^{n}\sum_{j=1}^{n}u_iu_j\rho_{i,j}', 
\end{equation}
for any real number $ u_1,u_2,\cdots,u_n $. This identity between two polynomials leads to
\begin{equation}\nonumber
\rho_{i,j}=\rho_{i,j}'. 
\end{equation}
Hence $ \textbf{N}=^d\textbf{N}' $.
\end{proof}

The next question is that what is the relation between two different ETS random vectors? Corollary \ref{The02} explains the relation between two different ETSD random vectors which the underlying vectors one of them is standard Gaussian and the other is any arbitrary Gaussian distribution.
\begin{corollary}\label{The02}
Let \textbf{X} be a centered elliptical tempered stable random vector in $ \mathbb{R}^n $ with standard normal underlying vector \textbf{N}.\footnote{The mean of centered random vector is that its mean is zero.} Then for any centered elliptical tempered stable random vector $ \textbf{X}' $ in $ \mathbb{R}^n $ there is a lower-triangular $n\times n $ matrix $ \Delta $ such that 
\begin{equation}\nonumber
\textbf{X}'=^d\Delta\textbf{X}.
\end{equation}
\end{corollary}
\begin{proof}
Let $ S_{t} $ be a tempered stable subordinator.\footnote{See \cite{Kim01}} An \textit{elliptical tempered stable process} $ \textbf{X}_{t} $ is defined by subordinating a n-dimensional multivariate Brownian motion $ \textbf{B}_{s} $ and tempered stable subordinator $ S_{t} $. \footnote{Since the ETS distribution is a subclass (symmetric) of MNTS distribution, so the definition of the elliptical tempered stable  process is well defined.} Therefore, the centered elliptical tempered stable process $ \textbf{X}_{t} $ is defined
as
\begin{equation}
\nonumber \textbf{X}_{t}=\textbf{B}_{S_{t}},
\end{equation}
where $ \textbf{B}_{s} $ denotes a centered Brownian motion with covariance matrix $ \boldsymbol{\Sigma} $. Since $ \textbf{B}_{s} $ can be degenerate as $ chol(\boldsymbol{\Sigma}) \textbf{W}_{s}$, we have 
\begin{equation}
\nonumber\textbf{X}_{t}=chol(\boldsymbol{\Sigma}) \textbf{W}_{s},
\end{equation}
where $ \textbf{W}_{t} $ is a n-dimensional process with each dimension being an independent standard Wiener process, and $ chol(\boldsymbol{\Sigma}) $ is lower triangle Cholesky factor. On the other hands, the covariance matrix $ \boldsymbol{\Sigma} $ can be decomposed into its vector of standard deviations $ \sigma $ and correlation matrix \textbf{P} , i.e.,
\begin{equation}
\nonumber\boldsymbol{\Sigma}=diag(\sigma)\textbf{P}diag(\sigma),
\end{equation} 
therefore, 
\begin{equation}\label{equ03}
\textbf{X}_{t}=diag(\sigma)chol(\textbf{P}) \textbf{W}_{s},
\end{equation}

Now by setting $ t=1 $ in equation (\ref{equ03}), the favorite has been obtained. 
\end{proof}

Now, we explain the link between fractional PDE, TID distribution, and particularly ETS distribution. In fact, we will introduce an initial value problem that the fundamental solution of it gives the PDF of TID distributions and particularly ETS distribution.
 
\begin{theorem}\label{The01}
Let $\textbf{X}_t$ be the position of a particle at time $t>0$, and $n-$dimentional Euclidean space $\mathbb{R}^n$. Let $P(\textbf{X},t)$ denote the density of $X_t$ where the vector $\textbf{X}=(X_1,X_2,\cdots,X_n)\in\mathbb{R}^n$. Then the fundamental solution of the fractional PDE as equation (\ref{equ30}) gives the PDF of the TID distribution
\begin{eqnarray}\label{equ30}
\frac{\partial P(\textbf{X},t)}{\partial t}=i\langle\textbf{m},\nabla P(\textbf{X},t)\rangle+\nabla_{R}^{\alpha}P(\textbf{X},t),
\end{eqnarray}
where $\textbf{m}\in\mathbb{R}^n$, $\nabla=(\frac{\partial}{\partial X_1},\frac{\partial}{\partial X_2},\cdots,\frac{\partial}{\partial X_n})$, and 
\begin{equation}
\nabla_{R}^{\alpha}f(\textbf{X})=F^{-1}\left( \left[ \int_{\mathbb{R}^n}\psi_{\alpha}(\left\langle \textbf{u},\textbf{X}\right\rangle )R(dx)\right] \hat{f}(\textbf{u}) \right), 
\end{equation}
where $ R $ and $ \psi $ are as the same as Definition \ref{TID:CF01}. And also, $ \hat{f}(\textbf{u}) $ and $ F^{-1} $ are the Fourier transform of $ f $ and inverse of Fourier transform, respectively. Furthermore, the initial condition of equation (\ref{equ30}) is $P(X_0=0)=1$.
\end{theorem}
\begin{proof} 
Before we start the proof, let $ f(\textbf{X},t):\mathbb{R}^n\times\mathbb{R}^+\longrightarrow\mathbb{C} $ be a function of $ \textbf{X}\in\mathbb{R}^n $ and $ t\in\mathbb{R}^+ $ such that $ f\in L^1(\mathbb{R}^n) $. The Fourier and inverse of Fourier transform will be given as
\begin{eqnarray}
\hat{f}(\textbf{u})=F(f(\textbf{X}))=\int\exp(-i\left\langle\textbf{X},\textbf{u} \right\rangle )f(\textbf{X})dX,\\
f(\textbf{X})=F^{-1}(\hat{f}(\textbf{u}))=\dfrac{1}{(2\pi)^n}\int\exp(i\left\langle\textbf{X},\textbf{u} \right\rangle )f(\textbf{u})dX,
\end{eqnarray}
The strategy for proof is the same as procedure was explained by Fallahgoul \textit{et al.}\footnote{It should be noted that some other fractional PDEs for multivariate stable and geometric stable distributions have been introduced. The difference between all kinds of these fractional PDEs backs to the structure of the L\`{e}vy measure of each of them. More information regarding univarite and multivarite fractional PDE regradless stable and geometric stable distribution can be found in \cite{Fallahgoul01, Fallahgoul02}.}

Taking the Fourier transform of equation (\ref{equ30}) and Fourier transform properties, we obtain
\begin{eqnarray}\label{equ31}
\frac{\partial \hat{P}(\textbf{u},t)}{\partial t}=i\langle\textbf{m},\textbf{u} \rangle\hat{P}(\textbf{u},t)+ \left[ \int_{\mathbb{R}^n}\psi_{\alpha}(\left\langle \textbf{u},\textbf{X}\right\rangle )R(dx)\right] \hat{P}(\textbf{u},t) ,
\end{eqnarray}
where $ \hat{P}(\textbf{u},t) $ is the Fourier transform of $ P(\textbf{X},t) $ with respect to $ \textbf{X} $. It should be noted that the initial condition $P(\textbf{X},0)= \delta(\textbf{X}) $ also converts to $ \hat{P}(\textbf{u},t)=1$. Now, equation (\ref{equ31}) is an initial value problem which all coefficient of it is constant, so we will have
\begin{equation}\label{equ32}
\hat{P}(\textbf{u},t)=\exp\left(i\langle\textbf{m},\textbf{u}t \rangle+t\int_{\mathbb{R}^n}\psi_{\alpha}(\left\langle \textbf{u},\textbf{X}\right\rangle )R(dx) \right). 
\end{equation}
Comparing equation (\ref{equ32}) and the CF of the TID distribution (equation (\ref{equ07})), we find that they are identical for a tempered stable distribution with $ \alpha\in(0,2) $. Therefore, the Green function solution of equation (\ref{equ30}) yields the entire PDF classes of the TID distribution. 
\end{proof}
The following corollary are the direct results of Theorem \ref{The01}.
\begin{corollary}
Let the operator $ \nabla_{R}^{\alpha} $ be
\begin{equation}
\nabla_{R}^{\alpha}f(\textbf{X})=F^{-1}\left( \left[ \int_{\mathbb{R}^n}\psi_{\alpha}^0(\left\langle \textbf{u},\textbf{X}\right\rangle )R(dx)\right] \hat{f}(\textbf{u}) \right), 
\end{equation}
in Theorem (\ref{The01}), and also $ \alpha\in(0,1) $. Then the Green function solution of equation (\ref{equ30}) yields the entire PDF classes of the TID distribution. 
\end{corollary}
\begin{corollary}\label{Cor:ETSD}
Let the operator $ \nabla_{R}^{\alpha} $ be
\begin{equation}
\nabla_{R}^{\alpha}f(\textbf{X})=F^{-1}\left( \left[  2^{-\alpha/2-1}\Gamma(-\dfrac{\alpha}{2}) \int_{\mathbb{R}^n} \left( _1F_1\left( -\dfrac{\alpha}{2},\dfrac{1}{2};\dfrac{-1}{2} \left\langle \textbf{u},\textbf{X}\right\rangle^2 \right)-1 \right) R(dx)\right] \hat{f}(\textbf{u}) \right), 
\end{equation}
in Theorem (\ref{The01}), where $ \alpha\in(0,2)=\frac{p}{q} $. Then the fundamental solution $ \hat{P}(\textbf{X},t) $ of equation
\begin{eqnarray}\label{equ33}
\frac{\partial P(\textbf{X},t)}{\partial t}=\nabla_{R}^{\alpha}P(\textbf{X},t),
\end{eqnarray}
with the initial condition $ P(\textbf{X},0)=\delta(\textbf{X}) $ is  the  density  of  the  ETSD.
\end{corollary}
\begin{remark}
In Corollary \ref{Cor:ETSD}, the location parameter is zero. But, based on the definition of symmetric TID about any real vector it can ba changed. In fact, in order to have the full form of stochastis representation definition of elliptical distribution, the location vector must be non-zero.
\end{remark}
\section{Analytic Approximation of Solution of Fractional PDE}
In this section, we derive the PDF approximation of multivariate tempered stable distributions by using the Homotopy Perturbation Method (hereafter HPM), Adomian Decomposition Method (hereafter ADM), and Variational Iteration Method (hereafter VIM).\footnote{See, for example \cite{He01}, \cite{Abdelrazec}, and \cite{He02}, among others.} It should be noted that the same procedure for getting the approximation of the PDF of fractional PDE regardless the ETS distribution can be applied.\footnote{It should be noted the details of computations have been omitted, and the reader can refer to \cite{Fallahgoul02} and \cite{Fallahgoul01} for more details. }  

\subsection{PDF Approximation of the TID Distribution Using the HPM}
To solve equation \ref{equ30} with initial condition $P{\rm (}X{\rm (0)=0)=1}$ using the HPM, we construct the following homotopy 
\begin{equation} \label{GrindEQ__5_1_} 
H{\rm (}p,V{\rm )=(1-}p{\rm )}\left(\frac{\partial V}{\partial t}{\rm -}\frac{\partial P_0}{\partial t}\right){\rm +}p\left(\frac{\partial V}{\partial t}{\rm +}<\textbf{m}{\rm ,}\nabla V>{\rm -}{\nabla }^{\alpha }_R\right){\rm =0.} 
\end{equation} 
Suppose the solution of equation \eqref{equ30} has the form 
\begin{equation} \label{GrindEQ__5_2_} 
V{\rm (}\textbf{X},t{\rm )=}p^0V_0{\rm +}p^{{\rm 1}}V_{{\rm 1}}{\rm +}p^{{\rm 2}}V_{{\rm 2}}{\rm +}\cdots {\rm +}p^{n{\rm -}{\rm 1}}V_{n{\rm -}{\rm 1}}{\rm +}\cdots {\rm .} 
\end{equation} 
Substituting equation \eqref{GrindEQ__5_2_} into equation \eqref{GrindEQ__5_1_}, and comparing coefficients of terms with identical powers of $p$, leads to 
\[p^0{\rm :\ \ \ \ \ \ \ \ }\frac{\partial V_0}{\partial t}{\rm -}\frac{\partial P_0}{\partial t}{\rm =0,}\] 
\[p^{{\rm 1}}{\rm :\ \ \ \ \ \ \ \ }\frac{\partial V_{{\rm 1}}}{\partial t}{\rm =}<\textbf{m}{\rm ,}\nabla V_0>{\rm +}{\nabla }^{\alpha }_RV_0,\] 
\[\vdots \] 
\[p^{n{\rm +1}}{\rm :\ \ \ \ \ \ \ \ }\frac{\partial V_{n{\rm +1}}}{\partial t}{\rm =}<\textbf{m}{\rm ,}\nabla V_n>{\rm +}{\nabla }^{\alpha }_RV_n.\] 

For simplicity, we take $V_0{\rm (}\textbf{X},t{\rm )=}P_0{\rm (}\textbf{X},t{\rm )}$. According to compared coefficients, we derive the following recurrent relation 
\[V_0{\rm (}\textbf{X},t{\rm )=}P_0{\rm (}\textbf{X},t{\rm ),}\] 
\[V_{{\rm 1}}{\rm (}\textbf{X},t{\rm )=}\int^t_0{}\left({\rm }<\textbf{m}{\rm ,}\nabla V_0>{\rm +}{\nabla }^{\alpha }_RV_0\right)dt{\rm ,\ \ \ \ \ \ \ \ }V_{{\rm 1}}{\rm (}\textbf{X}{\rm ,0)=0,}\] 
\[\vdots \] 
\[V_{n{\rm +1}}{\rm (}\textbf{X},t{\rm )=}\int^t_0{}\left({\rm }<\textbf{m}{\rm ,}\nabla V_n>{\rm +}{\nabla }^{\alpha }_RV_n\right)dt{\rm ,\ \ \ \ \ \ \ \ }V_{n{\rm +1}}{\rm (}\textbf{X}{\rm ,0)=0.}\] 

The solution is 
\[P\left(\textbf{X},t\right){\rm =}\mathop{{\rm lim}}_{p\to {\rm 1}}V\left(\textbf{X},t\right)\] 
\[{\rm =}V_0{\rm (}\textbf{X},t{\rm )+}V_{{\rm 1}}{\rm (}\textbf{X},t{\rm )+}V_{{\rm 2}}{\rm (}\textbf{X},t{\rm )+}\cdots {\rm +}V_{n{\rm +1}}{\rm (}\textbf{X},t{\rm )+}\cdots {\rm ,}\] 
\[{\rm =}\boldsymbol{\delta}(\textbf{X}) {\rm +}\int^t_0{}\left({\rm }<\textbf{m}{\rm ,}\nabla V_0>{\rm +}{\nabla }^{\alpha }_RV_0\right)dt{\rm +}\int^t_0{}\left({\rm }<\textbf{m}{\rm ,}\nabla V_{{\rm 1}}>{\rm +}{\nabla }^{\alpha }_RV_{{\rm 1}}\right)dt{\rm +}\cdots \] 
\[{\rm +}\int^t_0{}\left({\rm }<\textbf{m}{\rm ,}\nabla V_n>{\rm +}{\nabla }^{\alpha }_RV_n\right)dt{\rm +}\cdots {\rm .}\] 
Therefore, 
\[P{\rm (}\textbf{X},t{\rm )=}\boldsymbol{\delta}(\textbf{X}){\rm +}\sum^{\infty }_{k{\rm =0}}{}\left[\int^t_0{}\left({\rm }<\textbf{m}{\rm ,}\nabla V_k>{\rm +}{\nabla }^{\alpha }_RV_k\right)dt\right].\] 
\subsection{PDF Approximation of TID Distribution Using the ADM}
We will solve equation (\ref{equ30}) with initial condition $P{\rm (}\textbf{X}{\rm (0)=0)=1}$ using the ADM. To do so, we construct the following recurrence relation 
\[V_0{\rm (}\textbf{X},t{\rm )=}P_0{\rm (}\textbf{X},t{\rm )=}\boldsymbol{\delta}(\textbf{X}),\] 
\[V_{k{\rm +1}}{\rm (}\textbf{X},t{\rm )=}\int^t_0{}\left({\rm }<v{\rm ,}\nabla V_k>{\rm +}{\nabla }^{\alpha }_RV_k\right)dt{\rm ,\ \ \ \ \ \ \ \ }k\ge {\rm 0.}\] 

So, the solution is obtained as 
\[V_{{\rm 1}}{\rm (}\textbf{X},t{\rm )=}\int^t_0{}\left({\rm -}<v{\rm ,}\nabla V_0>{\rm +}{\nabla }^{\alpha }_RV_0\right)dt,\] 
\[\vdots \] 
\[V_{n{\rm +1}}\left(\textbf{X},t\right){\rm =}\int^t_0{}\left({\rm }\left\langle v{\rm ,}\nabla V_n\right\rangle {\rm +}{\nabla }^{\alpha }_RV_n\right)dt.\] 
\[\vdots \] 
Therefore,  
\begin{equation}\nonumber
P{\rm (}\textbf{X},t{\rm )=}\sum^{\infty }_{k{\rm =0}}{}V_k{\rm (}\textbf{X},t{\rm )=}\boldsymbol\delta {\rm (}\textbf{X}{\rm )+}\sum^{\infty }_{k{\rm =0}}{}\left[\int^t_0{}\left({\rm }<v{\rm ,}\nabla V_k>{\rm +}{\nabla }^{\alpha }_RV_k\right)dt\right].
\end{equation}
\subsection{PDF Approximation of TID Distribution Using the VIM}
The analytic approximation of the PDF of TID via VIM is considered. In fact, the fundamental solution of equation (\ref{equ30}) with initial condition $P{\rm (}\textbf{X}{\rm (0)=0)=1}$ via VIM is obtained. We set the following recurrence relation
\begin{equation}\label{equ08}
V_{n+1}(\textbf{X},t)=V_{n}(\textbf{X},t)+\lambda\int_{0}^{t}\left( \dfrac{\partial V_n(\textbf{X},s)}{\partial s}+(\left\langle \textbf{m},\nabla V_n(\textbf{X},s)\right\rangle -\nabla_R^{\alpha}V_n(\textbf{X},s)\right) ds,\qquad n=0,1,2,\cdots
\end{equation}
so
\begin{equation*}
\delta V_{n+1}(\textbf{X},t)=\delta V_{n}(\textbf{X},t)+\delta\int_{0}^{t}\lambda\left( \dfrac{\partial P_n(\textbf{X},s)}{\partial s}+(\left\langle \textbf{m},\nabla P_n(\textbf{X},s)\right\rangle -\nabla_R^{\alpha}P_n(\textbf{X},s)\right) ds=0.
\end{equation*}
Based on the stationary condition of equation (\ref{equ08}), i.e. $ \lambda+1=0, \lambda'=0 $, the Lagrange multiplier turns out to be $ \lambda=-1 $.

Now, by substituting $ \lambda=-1 $ in equation (\ref{equ08}), we get the following variational iteration formula
\begin{equation*}
V_{n+1}(\textbf{X},t)=V_{n}(\textbf{X},t)-\int_{0}^{t}\left( \dfrac{\partial V_n(\textbf{X},s)}{\partial s}+(\left\langle \textbf{m},\nabla V_n(\textbf{X},s)\right\rangle -\nabla_R^{\alpha}V_n(\textbf{X},s)\right) ds,\qquad n=0,1,2,\cdots
\end{equation*}
where $ V_0(\textbf{X},t)=P_0(\textbf{X},t)= \delta(\textbf{X})$.

Therefore
\begin{equation}
 \begin{tabular}{lllll}
 $P(\textbf{X},t)$&$=$&$\lim\limits_{n\longrightarrow\infty}V_n(\textbf{X},t)$\\&
 $=$&$V_0(\textbf{X},t)+\sum_{k=0}^{\infty}\left[ \int_{0}^{t}(-\left\langle \textbf{m},\nabla V_k\right\rangle +\nabla_R^{\alpha}V_k)dt\right] ,$\\&
  $=$&$\delta(\textbf{X})+\sum_{k=0}^{\infty}\left[ \int_{0}^{t}(-\left\langle \textbf{m},\nabla V_k\right\rangle +\nabla_R^{\alpha}V_k)dt\right].$
 \end{tabular}
\end{equation}
In this manner, the rest of the components of the VIM can be obtained. If we compute more terms, then we can show that $ P(\textbf{X}, t) $ is the TID's PDF with respect to $ \textbf{X} $, as $ \textbf{X}\sim TID_{\alpha}(R,t\textbf{m}) $. 
\section{Acknowledgment}
Hassan A. Fallahgoul acknowledges financial support from a Belgian Federal Science Policy Office (BELSPO) grant. He is a beneficiary of a mobility grant from the BELSPO co-funded by the Marie Curie Actions from the European Commission. Also, the comments and support of Professor David Veredas and Davy Paindaveine are gratefully acknowledged.

\bibliography{mybibfile}{}

\end{document}